\documentclass[a4paper,12pt]{article}
\usepackage{amscd}
\usepackage{amsmath,amsfonts,amssymb,amscd}
\usepackage{indentfirst,graphicx,epsfig}
\usepackage{graphicx,psfrag}
\input{epsf}

\newtheorem{df}{Definition}
\newtheorem{thm}{Theorem}

\newtheorem{lem}{Lemma}

\newtheorem{claim}{Claim}
\newtheorem{remark}{Remark}

\newenvironment {proof} {\noindent{\em Proof.}}{\hspace*{\fill}$\Box$\par\vspace{4mm}}

\def\qed{\hfill \nopagebreak\rule{5pt}{8pt}}

\title{\bf Some Motzkin-Straus type results for non-uniform
hypergraphs\footnote{Supported by NSFC and the ¡°973¡± program. } }
\author{
\small  Ran Gu$^1$, Xueliang Li$^1$, Yuejian Peng$^2$, Yongtang Shi$^1$\\
\small $^1$Center for Combinatorics and LPMC-TJKLC \\
\small Nankai University, Tianjin 300071, China \\
\small Email: guran323@163.com, lxl@nankai.edu.cn, shi@nankai.edu.cn
\\
\small $^2$College of Mathematics, Hunan University,\\
\small Changsha 410082, Hunan Province,
P.R. China\\
\small Email: ypeng1@163.com
\date{}}
\begin{document}
\maketitle
\begin{abstract}
A remarkable connection between the order of a maximum clique and
the Lagrangian of a graph was established by Motzkin and Straus in
1965. This connection and its extensions were applied in Tur\'{a}n
problems of graphs and uniform hypergraphs. Very recently, the study
of Tur\'{a}n densities of non-uniform hypergraphs has been motivated
by extremal poset problems. In this paper, we give some
Motzkin-Straus type results for non-uniform hypergraphs.
\\[2mm]
[2mm] \textbf{AMS Subject Classification (2010):} 05C65, 05D05
\end{abstract}

\section{Introduction}
In 1965, Motzkin and Straus \cite{MS} established a connection
between the order of a maximum clique and the Lagrangian of a graph,
which was used to give another proof of Tur\'{a}n's theorem. This
type of connection aroused interests in the study of Lagrangians of
uniform hypergraphs. Actually, the Lagrangian of a hypergraph has
been a useful tool in hypergraph extremal problems. Very recently,
the study of Tur\'{a}n densities of non-uniform hypergraphs has been
motivated by extremal poset problems; see \cite{GK, GL}. In this
paper, we intend to study the connection between the order of a
maximum clique and the Lagrangian of a non-uniform hypergraph.

A \emph{hypergraph} is a pair $H = (V,E)$ consisting of a vertex set
$V$ and an edge set $E$, where each edge is a subset of $V$. The set
$R(H)=\{|F|:F \in E\}$ is called the set of \emph{edge types} of
$H$. We also say that $H$ is an $R(H)$-graph. For example, if
$R(H)=\{1,3\}$, then we say that $H$ is a $\{1,3\}$-graph. If
all edges have the same cardinality $r$, then $H$ is an $r$-uniform
hypergraph, which is simply written as $r$-graph. A $2$-uniform
hypergraph is exactly a simple graph. A hypergraph is non-uniform
if it has at least two edge types. For any $r \in R(H)$, the
\emph{level hypergraph} $H^r$ is the hypergraph consisting of all
edges with $r$ vertices of $H$. We also use notation $E^r$ to denote
the set of all edges with $r$ vertices of $H$. We write $H_n^R$ for
a hypergraph $H$ on $n$ vertices with $R(H)=R$. For convenience, an
edge $\{i_1, i_2, \ldots, i_r\}$ in a hypergraph is simply written
as $i_1 i_2 \ldots i_r$ throughout the paper.

For an integer $n$, let $[n]$ denote the set $\{1, 2, \cdots, n\}$. The complete hypergraph $K_n^R$ is a hypergraph on  vertex set $[n]$  with edge set
$\bigcup\limits_{i \in R} {\left( {\begin{array}{*{20}{c}}
{[n]}\\
i
\end{array}} \right)}$. For example, $K_n^{\{ r\} }$ is the complete $r$-uniform
hypergraph on $n$ vertices.
$K_n^{[r]}$ is the non-uniform hypergraph with all possible edges of cardinality at
most $r$. Let $[n]^{R}$ represent the complete $R$-type hypergraph
on vertex set $[n]$. For example, $[n]^{\{1, 3\}}$ represents the complete $\{1, 3\}$-hypergraph
on vertex set $[n]$. We also let $[n]^{(r)}$ represent the complete $r$-uniform hypergraph
on vertex set $[n]$.

\begin{df}\label{def1}
For an $r$-uniform hypergraph $G$ with vertex set $\{1, 2,\cdots, n\}$, edge set $E(G)$
and a vector $\vec x = ({x_1}, \ldots ,{x_n}) \in {R^n}$, define
$$\lambda (G, \vec x) = \sum\limits_{{i_1}{i_2} \ldots {i_r} \in E(G)}
{{x_{{i_1}}}{x_{{i_2}}} \ldots {x_{{i_r}}}} {\rm{ }}.$$
\end{df}

\begin{df}\label{def2}
Let $S =\{\vec x = (x_1,\cdots, x_n) :
$ $\sum \limits_{i = 1}^n {{x_i}}  = 1,$ $ x_i \geq 0$ $ for$ $
i=1,2,\cdots,n\}$. The Lagrangian of $G$, denoted by $\lambda(G)$, is defined as
$$\lambda (G) = \max \{ \lambda (G,\vec x):\vec x \in S\}. $$
\end{df}

\noindent The value $x_i$ is called the weight of the vertex $i$ and
any vector $\vec x \in S$ is called a legal weighting. A weighting
$\vec y \in S$ is called an optimal weighting for $G$ if $\lambda(G,
\vec y)=\lambda(G)$.

Motzkin and Straus in \cite{MS} proved the following result for the
Lagrangian of a $2$-graph. It shows that the Lagrangian of a graph
is determined by the order of its maximum clique.

\begin{thm}\label{MS}\cite{MS}
If G is a $2$-graph in which a largest clique has order $t$, then,
\[\lambda (G) = \lambda \left( {{K_t}^{\{ 2\} }} \right) = \lambda
\left( {{{\left[ t \right]}^{(2)}}} \right) = \frac{1}{2}\left( {1 - \frac{1}{t}} \right).\]
\end{thm}
\noindent This connection provided another proof of Tur\'an's
theorem. More generally, the connection between Lagrangians and
Tur\'an densities can be used to give another proof of the
fundamental result of Erd\"s-Stone-Simonovits on Tur\'an densities
of graphs; see Keevash's survey paper \cite{Keevash}. In 1980's,
Sidorenko \cite{Sidorenko} and Frankl and F\"uredi \cite{FF}
developed the method of applying Lagrangians in determining
hypergraph Tur\'an densities. More applications of Lagrangians can
be found in \cite{FF1, Keevash}. Recently, the study of Tur\'an
densities of non-uniform hypergraphs has been motivated by the study
of extremal poset problems \cite{GK,GL}. A generalization of the
concept of Tur\'an density to a non-uniform hypergraph was given in
\cite{JL}.

In  \cite{Peng2}, the authors studied the Lagrangian of a $3$-graph
and proved the following result.

\begin{thm}\label{Peng2}\cite{Peng2}
 Let $m$ and $t$ be positive integers satisfying
$\binom{t}{3}\leq m\leq \binom{t}{3}+\binom{t-1}{2}$. Let $G$ be a $3$-graph with $m$ edges and contain a clique of order $t$. Then,
\[\lambda (G) = \lambda \left( {{[t]}^{( 3) }} \right).\]
\end{thm}

\noindent They pointed out that the upper bound
$\binom{t}{3}+\binom{t-1}{2}$ in this theorem is the best
possible. When $m = \binom{t}{3}+\binom{t-1}{2}+1$, let $H$ be the
$3$-graph with the vertex set $[t+1]$ and the edge set ${\left[ {t } \right]^{\left( 3 \right)}} \cup \left\{ {{i_1}{i_2}(t+1):{i_1}{i_2}
\in {{\left[ {t - 1} \right]}^{\left( 2 \right)}}} \right\} \cup
\left\{ {1t(t+1)} \right\}$. Take a legal weighting
$\vec x = (x_1,\cdots, x_n)$, where $x_1 = x_2 =\cdots= x_{t-1} =
\frac{1}{{t }}$ and $x_{t} = x_{t+1} = \frac{1}{{2t }}$. Then
$\lambda \left( {H} \right) \ge \lambda \left( {{H},\vec x} \right)
> \lambda \left( {{{\left[ {t } \right]}^{\left( 3 \right)}}}
\right)$.

Very recently, Peng et al. \cite{Peng1} introduced the Lagrangian of
a non-uniform hypergraph.

\begin{df}\label{def3}\cite{Peng1}
For a hypergraph $H_n^R$ and a vector $\vec x = ({x_1}, \ldots ,{x_n}) \in {R^n}$, define
\[\lambda'({H_n^R},\vec x){\rm{ = }}\sum\limits_{j \in R} {\left( {{j!}\sum\limits_{{i_1}{i_2} \ldots {i_j} \in {H^j}} {{x_{{i_1}}}{x_{{i_2}}} \ldots {x_{{i_j}}}} } \right)}. \]
\end{df}

\begin{df}\label{def4}\cite{Peng1}
Let $S =\{\vec x = (x_1,\cdots, x_n) :
$ $\sum \limits_{i = 1}^n {{x_i}}  = 1,$ $ x_i \geq 0$ $for$ $
i=1,2,\cdots,n\}$. The Lagrangian of $H_n^R$, denoted by $\lambda' (H_n^R)$, is defined as
\[\lambda' (H_n^R) = \max \{ \lambda' (H_n^R,  \vec x):\vec x \in S\}. \]
The value $x_i$ is called the weight of the vertex $i$ and any
vector $\vec x \in S$ is called a legal weighting. A weighting $\vec
y \in S$ is called an optimal weighting for $H$ if $\lambda'(H, \vec
y)=\lambda'(H)$.
\end{df}

\begin{remark}\label{rem1}
Consider the connection between Definition \ref{def2} and Definition
\ref{def4}. If $G$ is an $r$-uniform graph, then
$$\lambda'(G)=r!\lambda(G).$$
\end{remark}

In \cite{Peng1}, the authors proved the following generalization of
Motzkin-Straus result to $\{1,2\}$-graphs.

\begin{thm}\label{Peng1}\cite{Peng1}
If $H$ is a $\{1,2\}$-graph and the order of its maximum complete $\{1,2\}$-subgraph is $t$ (where $t\geq 2$), then,
\[\lambda' (H) = \lambda' \left( {{K_t}^{\{ 1,2\} }} \right) ={2 - \frac{1}{t}}.\]
\end{thm}

In this paper, we give a
Motzkin-Straus type result to $\{1,r\}$-graphs.  For any hypergragh (graph) $G$, denote the number of
its edges by $e(G)$.

\begin{thm}\label{th2}
Let $H$ be a $\{1,r\}$-graph. If both the order of its maximum complete
$\{1,r\}$-subgraph  and  the order of its maximum complete
$\{1\}$-subgraph are $t$, where $\displaystyle {t\geq \lceil {[r(r-1)-1]^{r-2} \over [r(r-1)]^{r-3}}\rceil }$,  then,
\[\lambda' (H) = \lambda' \left( {{K_t}^{\{ 1,r\} }} \right) ={1+  \frac{\prod_{i=1}^{r-1} (t-i)}{t^{r-1}}}.\]
\end{thm}

Furthermore, for $\{1,3\}$-graph, we give a result as follows.

\begin{thm}\label{th1}
Let $H$ be a $\{1,3\}$-graph. If the order of its maximum complete $\{1,3\}$-subgraph is $t$, where $t\geq 5$,  $H^3$ contains a maximum complete $3$-graph of order $s$, where $s\geq t$, and the number of edges in $H^3$ satisfies $\binom{s}{3}\leq e(H^3)\leq \binom{s}{3}+\binom{t-1}{2}$, then,
\[\lambda' (H) = \lambda' \left( {{K_t}^{\{ 1,3\} }} \right) ={1+  \frac{(t-1)(t-2)}{t^2}}.\]
\end{thm}

Notice that, if $r=3$, we require $t\geq 5$ in Theorems \ref{th2}
and \ref{th1}. In fact, for the case $t=3$ or $4$, it follows from the proof of Theorem \ref{th1},
Theorem \ref{th1} holds
when $s=t$. However, Theorem \ref{th1} fails to hold when
$t=3$ or $4$ and $s\geq t+1$. For $t=3$, $s\geq t+1$, let $G$ be the
$\{1,3\}$-graph with the vertex set $V(G)=[n]$ for some integer
$n\geq s$, and the edge set $E(G)=E^1\cup E^3$, where
$E^1=\{\{1\},\{2\},\{3\}\}$, $[s]^{(3)}\subseteq E^3$ and
$\binom{s}{3}\leq |E^3|\leq \binom{s}{3}+\binom{t-1}{2}$. Take a
legal weighting $\vec x = (x_1,\cdots, x_n)$, where $x_1 = x_2 = x_3
= 0.333$, $x_4=\cdots=x_s=\frac{0.001}{s-3}$,
$x_{s+1}=\cdots=x_n=0$, then $\lambda' \left( {G} \right) \ge
\lambda' \left( {{G},\vec x} \right) >{1+  \frac{(3-1)(3-2)}{3^2}}
=\lambda' \left( {{K_3}^{\{ 1,3\} }} \right)$.  This example also shows that Theorem \ref{th2} fails to hold when $t=3$ and $r=3$. For $t=4$, $s\geq
t+1$, let $G$ be a $\{1,3\}$-graph with the vertex set $V(G)=[n]$
for some integer $n\geq s$, and the edge set $E(G)=E^1\cup E^3$,
where $E^1=\{\{1\},\{2\}, \{3\},\{4\}\}$, $[s]^{(3)}\subseteq E^3$
and $\binom{s}{3}\leq |E^3|\leq \binom{s}{3}+\binom{t-1}{2}$. Take a
legal weighting $\vec x = (x_1,\cdots, x_n)$, where $x_1 = x_2 =
x_3=x_4 = 0.2498$, $x_5=\cdots=x_s=\frac{0.0008}{s-4}$,
$x_{s+1}=\cdots=x_n=0$, then $\lambda' \left( {G} \right) \ge
\lambda' \left( {{G},\vec x} \right) >{1+
\frac{(4-1)(4-2)}{4^2}}=\lambda' \left( {{K_4}^{\{ 1,3\} }}
\right)$.
 This example also shows that Theorem \ref{th2} fails to hold when $t=4$ and $r=3$.

The bound of $e(H^3)$ in Theorem \ref{th1} is necessary, and it is
also the best possible. When $e(H^3) =
\binom{s}{3}+\binom{t-1}{2}+1$, let $H$ be a $\{1,3\}$-graph with
the vertex set $[n]$ for some integer $n\geq s+1$, and the edge set
$E(H)=E^1\cup E^3$, where $E^1=\{\{1\},\cdots,\{ t\},\{ s+1\}\}$,
$E^3=\{{\left[ {s} \right]^{\left( 3 \right)}} \cup \left\{ {1t\left(
{s+1} \right)} \right\}\cup \{ {{i_1}{i_2}(s+1):{i_1}{i_2} \in
{{\left[ {t - 1} \right]}^{\left( 2 \right)}}} \} \}$. Then
 $[s+1]^{(3)}\nsubseteq E^3$, $|E^3|=
\binom{s}{3}+\binom{t-1}{2}+1$. Take a legal weighting $\vec x =
(x_1,\cdots, x_n)$, where $x_1 = x_2 =\cdots= x_{t-1} =
\frac{1}{t}$, $x_t=x_{s+1}=\frac{1}{2t}$ and the remaining
coordinates of $\vec x$ are equal to zero. Then $\lambda' \left( {H}
\right) \ge \lambda' \left( {{H},\vec x} \right) > \lambda' \left(
{{{\left[ {t} \right]}^{\left( 3 \right)}}} \right)$.

\section{Some preliminaries}
We will impose two additional conditions on any optimal legal weighting $\vec x = (x_1,\cdots, x_n)$ for an $R(H)$-graph $H$:

(i) $x_1\geq x_2 \geq\cdots \geq x_n \geq 0$,

(ii) $|\{j: x_j > 0\}|$ is minimal, i.e., if $\vec y$ is a legal weighting for $H$ satisfying $|\{j: y_j > 0\}|<|\{j: x_j > 0\}|$, then $\lambda' \left( {{H},\vec y} \right)<\lambda' \left( H \right)$.

Let $H=(V,E)$ be an $R(H)$-graph. For $r\in R(H)$, we will denote
the $(r-1)$-neighborhood of a vertex $i\in V$ by $E_i^r=\{A\in
V^{(r-1)}: A\cup \{i\} \in E^r\}$. Similarly, we  denote the
$(r-2)$-neighborhood of a pair of vertices $i, j \in V $ by
$E_{ij}^r=\{B\in V^{(r-2)}: B\cup \{i,j\} \in E^r\}$. We also denote
the complement of $E_i^r$ by $\overline E _i^r=\{A\in V^{(r-1)}:
A\cup \{i\} \in V^{(r)}\setminus E^r\}$, and define $\overline
E_{ij}^r=\{B\in V^{(r-2)}: B\cup \{i,j\} \in V^{(r)}\setminus
E^r\}$. For ease of notation, define $E_{i\setminus j}^r=E_i^r\cap
\overline E _j^r$. The following lemma gives some necessary
conditions of an optimal weighting for an $r$-graph $G$.

\begin{lem}\label{FR}\cite{FR}
Let $G = (V, E)$ be an $r$-graph and $\vec x = (x_1,\cdots, x_n)$ be
an optimal legal weighting for $G$ with $k (\leq n)$ positive
weights $x_1,\cdots, x_k$. Then for every $\{i, j \}\in [k]^{(2)}$,
(a) $\lambda(E_i^r , \vec x) =\lambda(E _j^r, \vec x) =
r\lambda(G)$, (b) there is an edge in $E$ containing both $i$ and
$j$.
\end{lem}

Consider the non-uniform hypergraph $H$, with Lagrangian
$\lambda'(H)$, in \cite{Peng1}, Peng et al. gave a similar result
for an $R(H)$-graph.

\begin{lem}\label{lem2.1}\cite{Peng1}
If $x_1\geq x_2 \geq\ldots \geq x_k > x_{k+1}=x_{k+2}=\ldots =x_n=0$ and $\vec x = (x_1,\cdots, x_n)$ be an optimal legal weighting of a hypergraph $H$, then, $\frac{{\partial \lambda '\left( {H,\vec x} \right)}}{{\partial {x_1}}} = \frac{{\partial \lambda '\left( {H,\vec x} \right)}}{{\partial {x_2}}} =  \cdots  = \frac{{\partial \lambda '\left( {H,\vec x} \right)}}{{\partial {x_k}}}$, and for every $\{i, j \}\in [k]^{(2)}$, there is an edge in $E$ containing both $i$ and $j$.
\end{lem}

In \cite{talbot},  Talbot introduced the definition of a left-compressed $r$-uniform hypergraph.  Let us generalize this concept to non-uniform hypergraphs.

Let $H=([n],E)$ be an $R(H)$-graph, where $n$ is a positive integer. For $e \in E$, and $i,j\in [n]$ with
$i<j$, then, define

\begin{equation*}
    {L_{ij}}\left( e \right)=
   \begin{cases}
   {(e\backslash \{ j\} ) \cup \{ i\} } &\mbox{if $i \notin e$ and $j\in e$,}\\
   e &\mbox{otherwise.}
   \end{cases}
  \end{equation*}
and
\begin{equation}\label{Lij}
\mathcal{L}_{ij}(E)=\{L_{ij}(e):e \in E\}\cup \{e:e,L_{ij}\left( e \right) \in E\}.
\end{equation}
Note that $|\mathcal{L}_{ij}(E)|=|E|$ from the definition of $\mathcal{L}_{ij}(E)$.

We say that $E$ ($H$) is \emph{left-compressed} if
$\mathcal{L}_{ij}(E)=E$ for every $1\leq i<j$.

\begin{lem}\label{lem1}
Let $H=([n],E)$ be an $R(H)$-graph,
$i,j\in [n]$ with $i<j$ and $\vec x = (x_1,\cdots, x_n)$ be an optimal legal weighting of $H$. Write $H_{ij}=([n],\mathcal{L}_{ij}(E))$. Then,
$$\lambda'(H, \vec x)\leq \lambda'(H_{ij}, \vec x).$$
\end{lem}

\begin{proof}
If $1 \notin R(H)$, then,
$$\lambda'(H_{ij}, \vec x)-\lambda'(H, \vec x)=\sum\limits_{r \in R(H)} {\sum\limits_{\scriptstyle e \in E^r,{L_{ij}}\left( e \right) \notin E^r\hfill\atop
\scriptstyle i \notin e,j \in e\hfill} {\lambda '(e\backslash \{ j\} ,\vec x)\left( {{x_i} - {x_j}} \right)} } ,$$
and if $1 \in R(H)$, then,
$$\lambda'(H_{ij}, \vec x)-\lambda'(H, \vec x)=
\sum\limits_{\scriptstyle r \in R(H)\hfill\atop
\scriptstyle r \ge 2\hfill} {\sum\limits_{\scriptstyle e \in {E^r},{L_{ij}}\left( e \right) \notin {E^r}\hfill\atop
\scriptstyle i \notin e,j \in e\hfill} {\lambda '(e\backslash \{ j\} ,\vec x)\left( {{x_i} - {x_j}} \right)} }  + \left( {{x_i} - {x_j}} \right)I,$$
where $I$ satisfies that $I=1$, if $i\notin E^1$
$j\in E^1$, and otherwise $I=0$.
Hence $\lambda'(H_{ij}, \vec x)-\lambda'(H, \vec x)$ is
nonnegative in any case, since $i<j$ implies that
$x_i\geq x_j$. So this lemma holds.
\end{proof}

\section{Proof of Theorem \ref{th2}}
Applying the  theory of Lagrangian multipliers,
it is easy to get that an optimal weighting  $\vec{x}$ for ${K_t}^{\{ 1,r\} }$ is given by $x_i=1/t$ for each $i$, $1\le i\le t$. So $\lambda'({{K_t}^{\{ 1,r\} }}) ={1+  \frac{\prod_{i=1}^{r-1} (t-i)}{t^{r-1}}}
$.  So we only need to prove $\lambda' (H) = \lambda' \left(
{{K_t}^{\{ 1,r\} }} \right)$. Since ${K_t}^{\{ 1,r\}}\subseteq H$,
clearly, $\lambda' (H) \geq \lambda' \left( {{K_t}^{\{ 1,r\} }}
\right)$. Thus, to prove Theorem \ref{th2}, it suffices to prove
that $\lambda' (H) \leq \lambda' \left( {{K_t}^{\{ 1,r\} }}
\right)$. Denote
$\lambda'_{\{t,\{1,r\}\}} =max\{\lambda' (G):$ $G$ is a
$\{1,r\}$-graph, $G$ contains a maximum complete subgraph
$K_t^{\{1,r\}}$ and a maximum complete subgraph $K_t^{\{1\}}\}$.  If
$\lambda'_{\{t,\{1,r\}\}}\leq  \lambda' \left( {{K_t}^{\{ 1,r\} }}
\right) $, then
$\lambda'(H)\leq  \lambda' \left( {{K_t}^{\{ 1,r\} }}
\right)$. Hence we can assume $H$
is an extremal hypergraph, i.e., $\lambda'
(H)=\lambda'_{\{t,\{1, r\}\}}$.  If $H$ is not left-compressed, performing a sequence of left-compressing operations (i.e. replace $E$ by $\mathcal{L}_{ij}(E)$ if $\mathcal{L}_{ij}(E)\neq E$), we will get a left-compressed $\{1, r\}$-graph $H'$ with the same number of edges. The condition that the order of a maximum complete $\{1\}$-subgraph of $H$ is $t$ guarantees that
both the order of a maximum $\{1, r\}$ complete subgraph of $H'$  and the order of a maximum $\{1\}$ complete subgraph of $H'$  are still $t$. By Lemma \ref{lem1}, $H'$ is an extremal graph as well.  So we can assume
that the edge set of $H$ is left-compressed, $H^1=[t]$ and $[t]^{(r)}\subseteq H^r$. Let $\vec x =
(x_1,\cdots, x_n)$ be an optimal legal weighting for $H$, where
$x_1\geq x_2 \geq\ldots \geq x_k > x_{k+1}=x_{k+2}=\ldots =x_n=0$. If $k\le t$, then $\lambda'(H)\le \lambda'([k]^{\{1, r\}})\le \lambda'([t]^{\{1, r\}})$.  So it suffices to show that $x_{t+1}=0$.

Let $1\le i\le t$. If $x_{t+1}>0$, then by Lemma \ref{lem2.1}, there exists $e\in H^r$ such that $\{i, t+1\}\subset e$ and $\frac{{\partial \lambda '\left( {H,\vec x}\right)}}{{\partial {x_i}}}=\frac{{\partial \lambda '\left( {H,\vec x}\right)}}{{\partial {x_{t+1}}}}$.

Recall that $i \in E^1$ and $t+1\notin E^1$, then,
\[\frac{{\partial \lambda '\left( {H,\vec x} \right)}}{{\partial {x_i}}} = 1 + r!\lambda \left( {E^r_{i\backslash (t+1)} ,\vec x} \right)+r!x_{t+1}\lambda \left( {E^r_{i(t+1)} ,\vec x} \right),\]
\[\frac{{\partial \lambda '\left( {H,\vec x} \right)}}{{\partial {x_{t+1}}}} =r!x_i\lambda \left( {E^r_{i(t+1)} ,\vec x} \right).\]
Let $A=r!\lambda \left( {E^r_{i\backslash (t+1)} ,\vec x} \right)$,
 and $C=r!\lambda \left( {E^r_{i(t+1)} ,\vec x} \right)$.
Thus, $x_i\geq
\frac{1}{C}+x_{t+1}$, with $0<C \le r!{(1-x_i-x_{t+1})^{r-2} \over (r-2)!}$. So
\begin{equation}\label{eq2}
x_i> \frac{1}{r(r-1)(1-x_i-x_{t+1})^{r-2}}+x_{t+1}.
\end{equation}

The above inequality clearly  implies that $x_i>{1 \over r(r-1)}$. Combining this with (\ref{eq2}), we have
\begin{equation}\label{eq1}
x_i> { [r(r-1)]^{r-3} \over [r(r-1)-1]^{r-2} }.
\end{equation}
Recall that $\displaystyle {t\geq \lceil {[r(r-1)-1]^{r-2} \over [r(r-1)]^{r-3}}\rceil }$, with the aid of (\ref{eq1}),
$\sum\limits_{i = 1}^t {{x_i}}  > 1 $,
a contradiction to the definition of legal weighting vectors.
So $x_{t+1}=0$.
The proof is thus complete.  \qed

\section{Proof of Theorem \ref{th1}}
As shown in Theorem \ref{th2}, $\lambda'({{K_t}^{\{ 1,3\} }}) =1+
\frac{(t-1)(t-2)}{t^2}$. So we only need to prove $\lambda' (H) =
\lambda' \left( {{K_t}^{\{ 1,3\} }} \right)$. Since ${K_t}^{\{
1,3\}}\subseteq H$, clearly, $\lambda' (H) \geq \lambda' \left(
{{K_t}^{\{ 1,3\} }} \right)$. Thus, to prove Theorem \ref{th1}, it
suffices to prove that $\lambda' (H) \leq \lambda' \left( {{K_t}^{\{
1,3\} }} \right)={1+  \frac{(t-1)(t-2)}{t^2}}$. This time we denote
$\mu_{\{t,s,m,\{1,3\}\}} =max\{\lambda' (G):$ $G$ is a
$\{1,3\}$-graph, $G$ contains a maximum complete subgraph
$K_t^{\{1,3\}}$, $G^3$ contains a maximum clique of order $s$ and
$e(G^3)=m$, where $\binom{s}{3}\leq m\leq
\binom{s}{3}+\binom{t-1}{2}$\}. If $\mu_{\{t,s,m,\{1,3\}\}}\leq
1+\frac{(t-1)(t-2)}{t^2}$, then $\lambda'(H)\leq 1+
\frac{(t-1)(t-2)}{t^2}$. Hence we can assume $H$ is an extremal
hypergraph, i.e., $\lambda' (H)=\mu_{\{t,s,m,\{1,3\}\}}$. Let $\vec
x = (x_1,\cdots, x_n)$ be an optimal legal weighting for $H$, where
$x_1\geq x_2 \geq\ldots \geq x_k > x_{k+1}=x_{k+2}=\ldots =x_n=0$.
Note that if $k\leq t$, then $\lambda'(H,\vec x)\leq \sum\limits_{i
= 1}^k {{x_i}} +\lambda'([k]^{(3)},\vec x)\leq
1+\lambda'([k]^{(3)})=1+  \frac{(k-1)(k-2)}{k^2}\leq 1+
\frac{(t-1)(t-2)}{t^2}$. Also, if $s=t$, then from Theorem
\ref{Peng2} and Remark \ref{rem1}, $\lambda'(H,\vec
x)\leq\sum\limits_{i = 1}^k {{x_i}}+\lambda'(H^3,\vec x)\leq
1+\lambda'(H^3)= 1+\frac{(s-1)(s-2)}{s^2}=1+\frac{(t-1)(t-2)}{t^2}$.
So in the sequel, we assume $k\geq t+1$ and $s\geq t+1$.

Since $e(H^3)\leq \binom{s}{3}+\binom{t-1}{2}$, there is a unique $K_s^{\{3\}}$ in $H^3$, otherwise, if $H^3$ contains two different $K_s^{\{3\}}$, then $e(H^3)\geq \binom{s}{3}+\binom{s-1}{2}$, a contradiction to the range of
$e(H^3)$. Let $\{i_1,\ldots,i_s\}$ be the vertex set of that  unique $K_s^{\{3\}}$ in $H^3$. We can assume there exists a unique vertex set $\{j_1,\ldots,j_t\}\subseteq \{i_1,\ldots,i_s\}$ such that $\{j_1,\ldots,j_t\}$ induces a $K_t^{\{1,3\}}$ in $H$.
Otherwise, since $e(H^3)\leq \binom{s}{3}+\binom{t-1}{2}$, there is a $K_t^{\{1,3\}}$ whose vertex set consists of a vertex $a \notin \{i_1,\ldots,i_s\}$ and $t-1$ vertices from $\{i_1,\ldots,i_s\}$, denote these $t-1$ vertices by $b_1,\ldots, b_{t-1}$. Notice that this $K_t^{\{1,3\}}$ is the unique  $K_t^{\{1,3\}}$ in $H$. Then we
take one vertex $b$ from $\{i_1,\ldots,i_s\}\setminus \{b_1,\ldots, b_{t-1}\}$, add a new $1$-edge $\{b\}$ to $H$, we can see that the new ${\{1,3\}}$-graph $H'$ satisfies the conditions of Theorem \ref{th1}, and $\lambda'(H')\geq\lambda'(H)$ since $H\subset H'$, which implies that $H'$ is also an extremal hypergraph.  Hence we can assume that there exists a unique vertex set $\{j_1,\ldots,j_t\}\subseteq \{i_1,\ldots,i_s\}$ such that $\{j_1,\ldots,j_t\}$ induces a $K_t^{\{1,3\}}$ in $H$.
Note that any vertex in $\{i_1,\ldots,i_s\}\setminus \{j_1,\ldots,j_t\} $ is not a $1$-edge in $H$.

Consider the relationship between the set $[k]$ and $\{i_1,\ldots,i_s\}$, we have three cases.

\noindent{\bf Case 1.} $[k]\subseteq \{i_1,\ldots,i_s\}$.

Denote $H_0$ the $\{1,3\}$-subgraph induced by $[k]$ in $H$, then $\lambda'(H_0)=\lambda'(H_0, \vec x)=\lambda'(H)$. We can see that $H_0$ satisfies the conditions of Theorem \ref{th2} ($r=3$), thus $\lambda'(H_0)=\lambda' \left( {{K_t}^{\{ 1,3\} }} \right) ={1+  \frac{(t-1)(t-2)}{t^2}}$, so $\lambda'(H)=\lambda' \left( {{K_t}^{\{ 1,3\} }} \right) ={1+  \frac{(t-1)(t-2)}{t^2}}$.

\noindent{\bf Case 2.} $[k]\cap \{i_1,\ldots,i_s\}=\emptyset$.

In this case, there are at most $\binom{t-1}{2}$ $3$-edges contributing nonzero value to $\lambda'(H,\vec x)$.
Let $H_0^3$ be the subgraph induced by $[k]$ in $H^3$, then $e(H_0^3)\leq \binom{t-1}{2}$. By adding some $3$-edges to
$H_0^3$, we can find a $3$-graph $G$ such that $H_0^3\subset G$,
$K_t^{\{3\}}\subset G$, and $e(G)\leq \binom{t}{3}+\binom{t-1}{2}$, by Theorem \ref{Peng2} and Remark \ref{rem1}, $\lambda'(H_0^3)\leq\lambda'(G)=\frac{(t-1)(t-2)}{t^2}.$
Hence $\lambda'(H,\vec x)\leq1+\lambda'(H^3,\vec x)=1+\lambda'(H_0^3,\vec x)\leq1+\lambda'(H_0^3)\leq1+\lambda'(G)=1+\frac{(t-1)(t-2)}{t^2}.$

\noindent{\bf Case 3.} $[k]\cap \{i_1,\ldots,i_s\}\neq \emptyset$,
and $[k]\nsubseteq \{i_1,\ldots,i_s\}$.

Let $|[k]\cap \{i_1,\ldots,i_s\}|=p$, and we will prove the
claim below.
\begin{claim}\label{claim2}
$|\{j:j \in \{j_1,\ldots,j_t\}, x_j>0\}|=min \{p,t\}$.
\end{claim}
\begin{proof}
Clearly, $|\{j:j \in \{j_1,\ldots,j_t\}, x_j>0\}|\le min \{p,t\}$.
If $|\{j:j \in \{j_1,\ldots,j_t\}, x_j>0\}|<min \{p,t\}$, then there exist two vertices $i,j$ such that $i\in \{i_1,\ldots,i_s\}\setminus \{j_1,\ldots,j_t\} $, $x_i>0$ and $j\in \{j_1,\ldots,j_t\}$, $x_j=0$. Denote $E^3_S$ the edge set of $K_s^{\{3\}}$ induced by
$\{i_1,\ldots,i_s\}$ in $H^3$. We construct a new ${\{1,3\}}$-graph $H'=([n],E')$, with $E'=(E\backslash A) \cup A'$, where $A$ is the edge set of all $3$-edges containing $i$ but not $j$ in $E\backslash E^3_S$, $A'$ is the edge set obtained from $A$ by replacing $i$ by $j$ for all $3$-edges in $A$. It is obvious that $|E'^3|=|E^3|$, $H'$ contains a $K_t^{\{1,3\}}$ and the order of maximum complete $3$-subgraph in $H'$ is still $s$, moreover,  we say that there is no $K_{t+1}^{\{1,3\}}$ in $H'$. Otherwise, there is a $K_{t+1}^{\{1,3\}}$ in $H'$, then the vertex
set of $K_{t+1}^{\{1,3\}}$ can not include vertices in $\{i_1,\ldots,i_s\}\setminus \{j_1,\ldots,j_t\}$, which indicates that there are at least $\binom{s}{3}+\binom{t}{2}$ $3$-edges in $H'$. Since $|E'^3|=|E^3|\leq\binom{s}{3}+\binom{t-1}{2}$, it is a contradiction. So the order of  maximum complete $\{1,3\}$-subgraph in $H'$ is still $t$.  We define a legal weighting $\vec {x'}$ for $H'$, such that $x'_l=x_l$, for $l\neq i, j$, and $x'_i=x_j=0$, $x'_j=x_i$. Then we can derive that $\lambda'(H',\vec {x'})- \lambda'(H,\vec x)\geq x_i>0$. This implies that $\lambda'(H')>\lambda'(H)$, a contradiction to the assumption of $H$.
\end{proof}
We still denote $H_0^3$ the subgraph induced by $[k]$ in $H^3$, and there are two subcases to consider.

{\bf Subcase 3.1.} $p\leq t$.

In this subcase, $H_0^3$ consists of a $K_p^{\{3\}}$ and at most $\binom{t-1}{2}$ other $3$-edges. Similarly to Case 2, by adding some $3$-edges, we can deduce that $\lambda(H_0^3)\leq\lambda(K_t^{\{3\}})$, then, $\lambda'(H^3,\vec x)=\lambda'(H_0^3,\vec x)\leq3!\lambda(H_0^3)\leq3!\lambda(K_t^{\{3\}})=
\frac{(t-1)(t-2)}{t^2}$,
so $\lambda'(H,\vec x)\leq1+\lambda'(H^3,\vec x)\leq 1+\frac{(t-1)(t-2)}{t^2}.$

{\bf Subcase 3.2.} $p\geq t+1$.

We  prove that we may assume for any $j\in  \{j_1,\ldots,j_t\}$, $i\in \{i_1,\ldots,i_s\}\setminus \{j_1,\ldots,j_t\}$,
\begin{equation}\label{eq3}
x_j\geq x_i,
\end{equation}
and
\begin{equation}\label{eq4}
\lambda \left( {E^3_{j\backslash i} ,\vec x} \right)\geq\lambda \left( {E^3_{i\backslash j} ,\vec x} \right)
\end{equation}
hold.

In fact, if $H$ dose not satisfy (\ref{eq3}) and (\ref{eq4}), through the following two steps, we will find a new $\{1,3\}$-graph $H^\ast$ and a new legal weighting vector $\vec {z}$ satisfying (\ref{eq3}) and (\ref{eq4}), and $H^\ast$ is an extremal hypergraph as well.

{\bf Step 1.} For every $i\in \{i_1,\ldots,i_s\}\setminus \{j_1,\ldots,j_t\}$, search for a vertex $j\in  \{j_1,\ldots,j_t\}$ satisfying  $E^3_{i\backslash j}\setminus E^3_{j\backslash i}\neq \emptyset$. If such a vertex exists (and if there is more than one such vertices, just take one of them), then for each $U\in E^3_{i\backslash j}\setminus E^3_{i\backslash j}$, replace
the $3$-edge $\{U\cup \{i\}\}$ by $\{U\cup \{j\}\}$. Check the value of $x_i$ and $x_j$, if $x_i>x_j$, then exchange the weight of these two vertices $i$, $j$.

Denote the new $\{1,3\}$-graph $H^\ast=([n],E^\ast)$ and the new legal weighting vector $\vec {y}$ obtained from Step 1.
We see that $|E^{\ast3}|=|E^3|$, the order of maximum complete $3$-subgraph in $H'$ is still $s$.  Similar to  the argument we used in Claim \ref{claim2}, there is no $K_{t+1}^{\{1,3\}}$ in $H^\ast$. Otherwise, there is a $K_{t+1}^{\{1,3\}}$ in $H^\ast$, then the vertex
set of $K_{t+1}^{\{1,3\}}$ can not include vertices in $\{i_1,\ldots,i_s\}\setminus \{j_1,\ldots,j_t\}$, which indicates that there are at least $\binom{s}{3}+\binom{t}{2}$ $3$-edges in $H^\ast$. Since $|E^{\ast 3}|=|E^3|\leq\binom{s}{3}+\binom{t-1}{2}$, it is a contradiction. So the order of  maximum complete $\{1,3\}$-subgraph in $H^\ast$ is still $t$.  Moreover, $H^\ast$ with  the weighting vector $\vec {y}$ satisfies (\ref{eq4}).

{\bf Step 2.} For every $i\in \{i_1,\ldots,i_s\}\setminus \{j_1,\ldots,j_t\}$ in $H^\ast$, search for a vertex $j\in  \{j_1,\ldots,j_t\}$ satisfying $y_i> y_j$ (if there are more than one such vertices, just take one of them). Then
exchange the weight of vertices $i$, $j$.

Denote the new legal weighting vector $\vec {z}$ for $H^\ast$ obtained after Step 2, then, clearly, $H^\ast$ with  the weighting vector $\vec {z}$ satisfies (\ref{eq3}) and (\ref{eq4}), besides, one can easily get that $\lambda'(H^\ast)\geq\lambda' (H^\ast,\vec z)\geq \lambda'(H, \vec x)=\lambda'(H)$. That implies $H^\ast$ is also an extremal hypergraph. Hence we can assume $H$ and its optimal weighting vector $\vec x$
satisfy that for any $j\in  \{j_1,\ldots,j_t\}$, $i\in \{i_1,\ldots,i_s\}\setminus \{j_1,\ldots,j_t\}$, (\ref{eq3}) and (\ref{eq4}) hold.

For any pair $i,j\in [k]$, if $i\in \{i_1,\ldots,i_s\}\setminus \{j_1,\ldots,j_t\}$, $j\in \{j_1,\ldots,j_t\}$, then,
\[\frac{{\partial \lambda '\left( {H,\vec x} \right)}}{{\partial {x_j}}} = 1 + 3!\lambda \left( {E^3_{j\backslash i} ,\vec x} \right)+3!x_i\lambda \left( {E^3_{ij} ,\vec x} \right),\]
\[\frac{{\partial \lambda '\left( {H,\vec x} \right)}}{{\partial {x_i}}} = 3!\lambda \left( {E^3_{i\backslash j} ,\vec x} \right)+3!x_j\lambda \left( {E^3_{ij} ,\vec x} \right).\]

Let $A=3!\lambda \left( {E^3_{j\backslash i} ,\vec x} \right)$,
$B=3!\lambda \left( {E^3_{i\backslash j} ,\vec x} \right)$,
$C=3!\lambda \left( {E^3_{ij} ,\vec x} \right)$. By Lemma \ref{lem2.1}, $1+A+x_iC=B+x_jC$.
With (\ref{eq4}), we have $A\geq B$, thus, $x_j>\frac{1}{C}+x_i$, with $0<C \le 6(1-x_i-x_j)$. So
\begin{equation}\label{eq5}
x_j\ge \frac{1}{6(1-x_i-x_j)}+x_i.
\end{equation}
The above inequality clearly  implies that $x_j>{1 \over 6}$. Combining this with (\ref{eq5}),
we have
\begin{equation}\label{eq111}
x_j> \frac{1}{5}+x_i.
\end{equation}
Since $p\geq t+1$, there exists a vertex $b\in[k]\cap \{i_1,\ldots,i_s\}\setminus \{j_1,\ldots,j_t\}$. If $t \geq 5$, then
$\sum\limits_{a \in E^1} {{x_a}} =\sum\limits_{a \in \{j_1,\ldots\,j_t\}} {{x_a}}>1+5x_b>1$, a contradiction to the definition of legal weighting vectors. Hence $t<5$,
which contradicts to the the condition $t \geq 5 $ in Theorem \ref{th1}.

Combining all these cases, the proof is thus complete.  \qed

\section{ Results for $\{1, r_2, \cdots, r_l\}$-graphs}

Applying similar method  used in the proof of Theorem \ref{th2}, we can obtain a result similar to Theorem \ref{th2} for  $\{1, r_2, \cdots, r_l\}$-graphs, where $l\ge 3$. Let us state this result.

\begin{thm}
Let $H$ be a $\{1, r_2, \cdots, r_l\}$-graph. If both the order of its maximum complete
$\{1, r_2, \cdots, r_l\}$-subgraph  and  the order of its maximum complete
$\{1\}$-subgraph are $t$, where $\displaystyle {t\geq f(r_2, \cdots, r_l)}$ for some function $f(r_2, \cdots, r_l)$,  then,
\[\lambda' (H) = \lambda' \left( {{K_t}^{\{ 1, r_2, \cdots, r_l\} }} \right).\]
\end{thm}

A formula for function $f(r_2, \cdots, r_l)$ could be given directly. But we omit the details. Let us skip the proof of the above result and give a detail proof for $\{1, 2, 3\}$-graphs.

\begin{thm}\label{th4}
Let $H$ be a $\{1, 2, 3\}$-graph. If both the order of its maximum complete
$\{1,2, 3\}$-subgraph  and  the order of its maximum complete
$\{1\}$-subgraph are $t$, where $t\ge 8$,  then,
\[\lambda' (H) = \lambda' \left( {{K_t}^{\{ 1, 2, 3\} }} \right) ={1+  \frac{t-1}{t}+\frac{(t-1)(t-2)}{t^2}}.\]
\end{thm}

\noindent {\bf \em Proof of  Theorem \ref{th4}. }The proof is
similar to the proof of Theorem \ref{th2}.   Applying the  theory of
Lagrangian multipliers, it is easy to get that an optimal weighting
$\vec{x}$ for ${K_t}^{\{ 1, 2, 3\} }$ is given by $x_i=1/t$ for each
$i$, $1\le i\le t$. So $\lambda' \left( {{K_t}^{\{ 1, 2, 3\} }}
\right) ={1+  \frac{t-1}{t}+\frac{(t-1)(t-2)}{t^2}}$.
  So we only need to prove $\lambda' (H) = \lambda' \left(
{{K_t}^{\{ 1, 2, 3\} }} \right)$. Since ${K_t}^{\{ 1, 2, 3\}}\subseteq H$,
clearly, $\lambda' (H) \geq \lambda' \left( {{K_t}^{\{ 1, 2, 3\} }}
\right)$. Thus, to prove Theorem \ref{th4}, it suffices to prove
that $\lambda' (H) \leq \lambda' \left( {{K_t}^{\{ 1, 2, 3\} }}
\right)$. Denote
$\lambda'_{\{t,\{1, 2, 3\}\}} =max\{\lambda' (G):$ $G$ is a
$\{1, 2, 3\}$-graph, $G$ contains a maximum complete subgraph
$K_t^{\{1, 2, 3\}}$ and a maximum complete subgraph $K_t^{\{1\}}\}$.  If
$\lambda'_{\{t,\{1, 2, 3\}\}}\leq  \lambda' \left( {{K_t}^{\{ 1, 2, 3\} }}
\right) $, then
$\lambda'(H)\leq  \lambda' \left( {{K_t}^{\{ 1, 2, 3\} }}
\right)$. Hence we can assume $H$
is an extremal hypergraph, i.e., $\lambda'
(H)=\lambda'_{\{t,\{1, 2, 3\}\}}$.  If $H$ is not left-compressed, performing a sequence of left-compressing operations (i.e. replace $E$ by $\mathcal{L}_{ij}(E)$ if $\mathcal{L}_{ij}(E)\neq E$), we will get a left-compressed $\{1, 2, 3\}$-graph $H'$ with the same number of edges. The condition that the order of a maximum complete $\{1\}$-subgraph of $H$ is $t$ guarantees that
 both the order of a maximum $\{1, 2, 3\}$ complete subgraph of $H'$  and the order of a maximum $\{1\}$ complete subgraph of $H'$  are still $t$. By Lemma \ref{lem1}, $H'$ is an extremal graph as well.  So we can assume
that the edge set of $H$ is left-compressed, $H^1=[t]$, $[t]^{(2)}\subseteq H^2$ and $[t]^{(3)}\subseteq H^3$. Let $\vec x =
(x_1,\cdots, x_n)$ be an optimal legal weighting for $H$, where
$x_1\geq x_2 \geq\ldots \geq x_k > x_{k+1}=x_{k+2}=\ldots =x_n=0$. If $k\le t$, then $\lambda'(H)\le \lambda'([k]^{\{1, 2, 3\}})\le \lambda'([t]^{\{1, 2, 3\}})$.  So it suffices to show that $x_{t+1}=0$.

Let $1\le i\le t$. If $x_{t+1}>0$, then by Lemma \ref{lem2.1}, there exists $e\in E$ such that $\{i, t+1\}\subset e$ and $\frac{{\partial \lambda '\left( {H,\vec x}\right)}}{{\partial {x_i}}}=\frac{{\partial \lambda '\left( {H,\vec x}\right)}}{{\partial {x_{t+1}}}}$. Let $\lambda ( E^2_{i(t+1)} ,\vec x)=1$, if $i(t+1)\in E^2$, and let $\lambda ( E^2_{i(t+1)} ,\vec x)=0$, if $i(t+1)\notin E^2$. Recall that $i \in E^1$ and $t+1\notin E^1$,  then,
\begin{eqnarray*}
\frac{{\partial \lambda '\left( {H,\vec x} \right)}}{{\partial {x_i}}} &= &1 +2!\lambda \left( {E^2_{i\backslash (t+1)} ,\vec x} \right)+2!x_{t+1}\lambda \left( {E^2_{i(t+1)} ,\vec x} \right)+ 3!\lambda \left( {E^3_{i\backslash (t+1)} ,\vec x} \right)\\
&\null&+3!x_{t+1}\lambda \left( {E^3_{i(t+1)} ,\vec x} \right);
\end{eqnarray*}
\begin{eqnarray*}
\frac{{\partial \lambda '\left( {H,\vec x} \right)}}{{\partial {x_{t+1}}}} = 2!x_{i}\lambda \left( {E^2_{i(t+1)} ,\vec x} \right)+
3!x_i\lambda \left( {E^3_{i(t+1)} ,\vec x} \right).
\end{eqnarray*}
Let
$A=2!\lambda \left( {E^2_{i\backslash t+1} ,\vec x} \right)+ 3!\lambda \left( {E^3_{i\backslash t+1} ,\vec x} \right)$,
and $C=2!\lambda \left( {E^2_{i(t+1)} ,\vec x} \right)+3!\lambda \left( {E^3_{i(t+1)} ,\vec x} \right)$.
By Lemma \ref{lem2.1}, $1+A+x_{t+1}C=x_iC$.
Thus, $x_i\geq
\frac{1}{C}+x_{t+1}$, with $0<C \le 2+6(1-x_i-x_{t+1})$. Hence
\begin{equation}\label{eq22}
x_i> \frac{1}{2+6(1-x_i-x_{t+1})}+x_{t+1}.
\end{equation}

The above inequality clearly  implies that $x_i>{1 \over 8}$. Combining this with (\ref{eq22}), we have
\begin{equation}\label{eq11}
x_i> {4 \over 29 }.
\end{equation}
Recall that $t\ge 8$, with the aid of (\ref{eq11}),
$\sum\limits_{i = 1}^t {{x_i}}  > 1 $,
a contradiction to the definition of legal weighting vectors.
So $x_{t+1}=0$.
The proof is thus complete.  \qed


\begin{thebibliography}{1}

\bibitem{FF}
P. Frankl, Z. F\"uredi, Extremal problems and the Lagrange
function of hypergraphs, {\it Bulletin Institute Math. Academia
Sinica} {\bf 16}(1988), 305--313.

\bibitem{FF1}
P. Frankl, Z. F\"uredi, Extremal problems whose solutions are the
blow-ups of the small Witt-designs, {\it J. Combin. Theory Ser. A.}
{\bf 52}(1989), 129--147.

\bibitem{FR}
P. Frankl, V. R\"{o}dl, Hypergraphs do not jump, {\it Combinatorica}
{\bf 4}(1984), 149--159.

\bibitem{Peng1}
Y. Peng, H. Peng, Q. Tang, C. Zhao, An extension of Motzkin-Straus
Thorem to non-uniform hypergraphs and its applications, preprint.

\bibitem{Peng2}
Y. Peng, C. Zhao, A Motzkin-Straus type result for $3$-uniform
hypergraphs, {\it Graphs Combin.} {\bf 29}(2013), 681--694.

\bibitem{MS}
T. Motzkin, E. Straus, Maxima for graphs and a new proof of a
theorem of Tur\'{a}n, {\it Canad. J. Math.} {\bf 17}(1965),
533--540.


\bibitem{GK}
J. Griggs, G. Katona, No four subsets forming an N, {\it J. Combin.
Theory Ser. A.} {\bf 115}(2008), 677--685.

\bibitem{GL}
J. Griggs, L. Lu, On families of subsets with a forbidden subposet,
{\it Comb. Probab. Comput.} {\bf 18}(2009), 731--748.

\bibitem{Keevash}
P. Keevash, Hypergrah Tur¡äan problems, http://www.maths.qmul.ac.uk/
keevash/papers/turansurvey. pdf.

\bibitem{JL}
T. Johston, L. Lu, Tur\'an problems on non-uniform hypergraphs,
submitted.

\bibitem{Sidorenko}
A. Sidorenko, The maximal number of edges in a homogeneous
hypergraph containing no prohibited subgraphs, {\it Math Notes} {\bf
41}(1987), 247--259. Translated from Mat. Zametki.

\bibitem{talbot}
J. Talbot, Lagrangians of hypergraphs, {\it Comb. Probab. Comput.} {\bf 11} (2002), 199--216.

\end{thebibliography}
\end{document}